\documentclass{amsart}

\usepackage{amsmath,amsfonts,amssymb,amscd}

\def\Q{{\mathbb Q}}
\def\Z{{\mathbb Z}}
\def\C{{\mathbb C}}

\def\F{{\mathbb F}}
                               \def\ST{{\mathbf S}}

\def\CC{{\mathcal C}}

\def\Gal{\mathrm{Gal}}

\def\End{\mathrm{End}}
\def\Aut{\mathrm{Aut}}

\def\cl{\mathrm{cl}}

                        \def\II{\mathrm{Id}}
\def\supp{\mathrm{supp}}

\def\fchar{\mathrm{char}}

\def\Sp{\mathrm{Sp}}
                           \def\Gp{\mathrm{Gp}}

\def\dim{\mathrm{dim}}

                             \def\cl{\mathrm{cl}}
                             \def\supp{\mathrm{supp}}
                               \def\sep{\mathrm{sep}}

     \def\W{{\mathfrak W}}

      \def\sep{\mathrm{sep}}

                          \def\rr{{\mathfrak r}}
                          \def\RR{{\mathfrak R}}
\def\a{{\mathfrak a}}
                               \def\b{{\mathfrak b}}

\newtheorem{thm}{Theorem}[section]
\newtheorem{lem}[thm]{Lemma}
\newtheorem{cor}[thm]{Corollary}

\theoremstyle{definition}

\newtheorem{ex}[thm]{Example}

           \newtheorem{rem}[thm]{Remark}

\title{Division by $2$ on hyperelliptic curves and jacobians}
\author {Yuri G. Zarhin}

\address{Pennsylvania State University, Department of Mathematics, University Park, PA 16802, USA}

\email{zarhin@math.psu.edu}

\begin{document}\date{}

\maketitle

\section{Introduction}

Let $K$ be an algebraically closed field of characteristic different from $2$.
If $n$ and $i$  are positive integers and $\mathbf{r}=\{r_1, \dots, r_n\}$ is a sequence of $n$ elements $r_i\in K$ then we write
$$\mathbf{s}_i(\mathbf{r})=\mathbf{s}_i(r_1, \dots, r_n)\in K$$
for the $i$th basic symmetric function in $r_1, \dots, r_n$. If we put $r_{n+1}=0$ then $\mathbf{s}_i(r_1, \dots, r_n)=\mathbf{s}_i(r_1, \dots, r_n, r_{n+1})$.

 Let $g \ge 1$ be an integer.
Let $\CC$ be the smooth projective model of the smooth affine plane $K$-curve
$$y^2=f(x)=\prod_{i=1}^{2g+1}(x-\alpha_i)$$
where $\alpha_1,\dots, \alpha_{2g+1}$ are {\sl distinct} elements of $K$. It is well known that $\CC$ is a genus $g$ hyperelliptic curve over $K$ with precisely one {\sl infinite} point, which we denote by $\infty$.  In other words,
$$\CC(K)=\{(a,b)\in K^2\mid  b^2=\prod_{i=1}^{2g+1}(a-\alpha_i)\}\bigsqcup \{\infty\} .$$
Clearly, $x$ and $y$ are nonconstant rational functions on $\CC$, whose only pole is $\infty$. More precisely, the polar divisor of $x$ is $2 (\infty)$ and the polar divisor of $y$ is $(2g+1)(\infty)$. The zero divisor of $y$ is $\sum_{i=1}^{2g+1} (\W_i)$
where
$$\W_i=(\alpha_i,0) \in \CC(K) \ \forall i=1, \dots , 2g+1.$$
 We write $\iota$ for the hyperelliptic involution 
$$\iota: \CC \to \CC,  (x,y)\mapsto (x,-y), \ \infty \mapsto\infty.$$
The set  of fixed points of $\iota$ consists of $\infty$ and all $\W_i$.
It is well known that for each $P \in \CC(K)$ the divisor $(P)+\iota(P)-2(\infty)$ is principal. More precisely, if $P=(a,b)\in \CC(K)$ then $(P)+\iota(P)-2(\infty)$ is the divisor of the rational function $x-a$ on $C$. If $D$ is a divisor on $\CC$ then we write $\supp(D)$ for its {\sl support}, which is a finite subset of $\CC(K)$.

We write $J$ for the jacobian of $\CC$, which is a $g$-dimensional abelian variety over $K$.  If $D$ is a degree zero divisor on $\CC$ then we write $\cl(D)$ for its linear equivalence class, which is viewed as an element of $J(K)$.
We will identify $\CC$ with its image in $J$ with respect to the canonical regular map $\CC \hookrightarrow J$ under which  $\infty$ goes to the zero of group law on $J$. In other words, a point $P \in \CC(K)$ is identified with  $\cl((P)-(\infty))\in J(K)$. Then the action of  $\iota$ on  $\CC(K)\subset J(K)$ coincides with multiplication by $-1$ on $J(K)$.
In particular, the list of points of order  2 on $\CC$ consists of  all $\W_i$.

Recall \cite[Sect. 13.2, p. 411]{Wash} that if $D$ is an effective divisor of (nonnegative) degree $m$, whose support does {\sl not} contain $\infty$, then the degree zero divisor $D-m(\infty)$ is called {\sl semi-reduced} if it enjoys the following properties.

\begin{itemize}
\item
If $\W_i$ lies in $\supp(D)$ then it appears in $D$ with multiplicity 1.
\item
If a  a point $Q$ of $\CC(K)$ lies in $\supp(D)$ and does not coincide with any of $\W_i$  then $\iota(P)$ does {\sl not} lie in $\supp(D)$.
\end{itemize}
If, in addition, $m \le g$ then $D-m(\infty)$ is called {\sl reduced}.

It is known (\cite[Ch. 3a]{Mumford}, \cite[Sect. 13.2, Prop. 3.6 on p. 413]{Wash}) that for each $\a \in J(K)$ there exist {\sl exactly one}  nonnegative  $m$ and  (effective) degree $m$ divisor  $D$  such that the degree zero divisor $D-m(\infty)$ is {\sl reduced} and  $\cl(D-m(\infty))=\a$.  (E.g.,, the zero divisor with $m=0$ corresponds to $\a=0$.)  If 
$$m\ge 1, \ D=\sum_{j=1}^m(Q_j)\ \mathrm{ where }  \ Q_j=(b_j,c_j) \in \CC(K) \ \forall \ j=1, \dots , m$$
(here $Q_j$ do {\sl not} have to be distinct)
then the corresponding
$$\a=\cl(D-m(\infty))=\sum_{j=1}^m Q_j \in J(K).$$
The {\sl Mumford's representation} (\cite[Sect. 3.12]{Mumford}, \cite[Sect. 13.2, pp. 411--415, especially, Prop. 13.4, Th. 13.5 and Th. 13.7]{Wash} 
of $\a \in J(K)$ is 
is the pair $(U(x),V(x))$ of polynomials $U(x),V(x)\in K[x]$ such that 
$$U(x)=\prod_{j=1}^r(x-a_j)$$
is a degree $r$ monic polynomial while $V(x)$ has degree $m<\deg(U)$, the polynomial  $V(x)^2-f(x)$ is divisible by $U(x)$, and
$D-m(\infty)$ coincides with the {\sl gcd} (i.e., with the  minimum) of the divisors of rational functions $U(x)$ and $y-V(x)$ on $\CC$.  This implies that each $Q_j$ is a zero of $y-V(x)$, i.e.,
$$b_j=V(a_j), \ Q_j=(a_j,V(a_j))\in \CC(K) \ \forall \ j=1, \dots m.$$
  Such a pair always exists, it is unique, and (as we've just seen) uniquely determines not only $\a$ but also divisors $D$  and $D-m(\infty)$.
(The case $\alpha=0$ corresponds to $m=0, D=0$ and the pair $(U(x)=1, V(x)=0)$.) 
 
Conversely, if $U(x)$ is a  monic polynomial of degree $m\le g$ and $V(x)$ a polynomial such that $\deg(V)<\deg(U)$ and $V(x)^2-f(x)$ is divisible by $U(x)$ then there exists exactly one  $\a=\cl(D-m(\infty))$ where $D-m(\infty)$ is a reduced divisor such that $(U(x),V(x)$ is the Mumford's representation  of 
$\cl(D-m(\infty))$.

Let $P=(a,b)$ be a $K$-point on $\CC$, i.e.,
$$a,b \in K, \ b^2=f(a)=\prod_{i=1}^n(a-\alpha_i).$$
The aim of this note is to divide explicitly $P$  by $2$ in $J(K)$, i.e., to give explicit formulas for the {\sl Mumford's representation} 
 of {\sl all} $2^{2g}$  divisor classes $\cl(D-g(\infty))$  such that $2D+\iota(P)$ is linearly equivalent to  $(2g+1)\infty$, i.e.,  
$$2\cl(D-g(\infty))=P \in \CC(K)\subset J(K).$$
(It turns out that each such $D$ has degree $g$ and its support does not contain any of   $\W_i$.)

The paper is organized as follows. In Section \ref{divisors} we obtain auxiliary results about divisors on hyperelliptic curves. In particular, we prove (Theorem \ref{notheta}) that if $g>1$ then the only point of $\CC(K)$ that is divisible by two in $\CC(K)$ (rather than in $J(K)$) is $\infty$ (of course, if $g>1$). We also prove that $\CC(K)$ does {\sl not} contain points of order $n$ if $2<n\le 2g$. 
In Section \ref{division} we describe  explicitly for a given $P=(a,b)\in \CC(K)$ the Mumford's representation of $2^{2g}$ divisor classes $\cl(D-g(\infty))$ such that $D$ is an effective degree $g$ reduced divisor on $\CC$ and
$$2\cl(D-g(\infty))=P \in \CC(K)\subset J(K).$$
The description is given in terms of square roots $\sqrt{a-\alpha_i}$'s ($1\le i\le 2g+1$), whose product is $-b$. (There are exactly $2^{2g}$ choices of such square roots.) In Section \ref{rat} we discuss the {\sl rationality questions}, i.e., the case when
$f(x),\CC,J$ and $P$ are defined over a subfield $K_0$ of $K$ and ask when dividing $P$ by $2$ we get a point of $J(K_0)$. 

Sections \ref{torsion} and \ref{symplectic} deal with torsion points on  certain naturally arised subvarieties of $J$ containg $\CC$. 
In particular, we discuss the case of a {\sl generic hyperelliptic curve} in characteristic zero, 
using as a starting point results of B. Poonen - M. Stoll \cite{PoonenStoll} and of J. Yelton \cite{Yelton}.
 Our approach is based on ideas of J.-P. Serre \cite{SerreMFV} and F. Bogomolov \cite{Bogomolov}.

This paper is a follow up of \cite{Zarhin,BZ} where the (more elementary) case of elliptic curves is discussed.  

{\bf Acknowledgements}. I am deeply grateful to Bjorn Poonen for helpful stimulationg discussions.
This work was partially supported by a grant from the Simons Foundation (\#246625 to Yuri Zarkhin). I've started to write this paper
 during my stay in May-June 2016 at the Max-Planck-Institut f\"ur Mathematik (Bonn, Germany),
whose hospitality and support are gratefully acknowledged.

\section{Divisors on hyperelliptic curves}
\label{divisors}

\begin{lem}[Key Lemma]
\label{key}
Let $D$ be an effective divisor on $\CC$ of  degree $m>0$ such that $m \le 2g+1$ and $\supp(D)$ does not contain $\infty$. Assume that the divisor $D-m(\infty)$ is principal.

\begin{enumerate}
\item[(1)]
Suppose that $m$ is odd.
 Then:
 
 \begin{itemize}
 \item[(i)]
  $m=2g+1$ and there exists exactly one polynomial $v(x)\in K[x]$ such that  the divisor of $y-v(x)$ coincides with $D-(2g+1)(\infty)$. In addition, $\deg(v)\le g$.
    \item[(ii)]
    If $\W_i$ lies in  $\supp(D)$ then it appears in $D$ with multiplicity 1.
    \item[(iii)]
    If $b$ is a nonzero element of $K$ and a $K$-point $P=(a,b) \in \CC(K)$ lies in  $\supp(D)$ then $\iota(P)=(a,-b)$ does not lie in  $\supp(D)$.
  \end{itemize}
  \item[(2)]
  Suppose that $m=2d$ is even. Then there exists exactly one monic degree $d$ polynomial $u(x)\in K[x]$ such that  the divisor of $v(x)$ coincides with $D-m(\infty)$.  In particular, every point $Q \in \CC(K)$ appears in $D-m(\infty)$ with the same multiplicity as $\iota(Q)$.
  \end{enumerate}
\end{lem}

\begin{proof}
Leh $h$ be a rational function on $\CC$, whose divisor coincides with $D-m(\infty)$. Since $\infty$ is the only pole of $h$, the function $h$ is a polynomial in $x,y$ and therefore may be presented as
$$h=s(x)y-v(x),   \ \mathrm{ with } \ u,v \in K[x].$$
If $s=0$ then $h$ has at $\infty$ the pole of even order $2\deg(v)$ and therefore $m=2\deg(v)$. 

Suppose that
 $s \ne 0$. Clearly, $s(x)y$ has at $\infty$ the pole of odd order $2\deg(s)+(2g+1) \ge (2g+1)$. So, the orders of the pole for $s(x)y$ and $v(x)$ are distinct, because they have different parity and therefore the order $m$ of the pole of $h=s(x)y-v(x)$ coincides with
$\max(2\deg(s)+(2g+1), 2\deg(v))\ge 2g+1$.
This implies that $m=2g+1$; in particular, $m$ is even. It follows that $m$ is {\sl even} if and only if $s(x)=0$, i.e., $h=-v(x)$; in addition, $\deg(v)\le (2g+1)/2$, i.e., $\deg(v)\le g$.  In order to finish the proof of (2), it suffices to divide $-v(x)$ by its leading coefficient and denote the ratio by $u(x)$. (The uniqueness of monic $u(x)$ is obvious.)

Let us prove (1).   
Since $m$ is {\sl odd},
$$m=2\deg(s)+(2g+1)>2\deg(v).$$
 Since $m \le 2g+1$,  we obtain that $\deg(s)=0$, i.e., $s$ is a nonzero element of $K$ and $2\deg(v)< 2g+1$.  The latter inequality means that $\deg(v)\le g$. Dividing $h$ by the constant $s$, we may and will assume that $s=1$ and therefore $h=y-v(x)$ with 
 $$v(x)\in K[x], \ \deg(v) \le g.$$
  This proves (i). (The uniqueness of $v$ is obvious.)
  The assertion (ii)  is contained in Proposition 13.2(b) on  pp. 409-10 of \cite{Wash}.   
 In order to prove (iii),  we just follow arguments on p. 410 of \cite{Wash} (where it is actually proven). Notice that our $P=(a,b)$ is a zero of $y-v(x)$, i.e. $b-v(a)=0$. Since, $b\ne 0$, $v(a)=b \ne 0$ and $y-v(x)$ takes on at $\iota(P)=(a,-b)$ the  value $-b-v(a)=-2b \ne 0$.  This implies that $\iota(P)$ is {\sl not} a zero of $y-v(x)$, i.e., $\iota(P)$ does not lie in $\supp(D)$.
\end{proof}

\begin{rem}
 Lemma \ref{key}(1)(ii,iii) asserts that if $m$ is odd
the divisor $D-m(\infty)$ is {\sl semi-reduced}. See \cite[the penultimate paragraph on p. 411]{Wash}.  
\end{rem} 

\begin{cor}
\label{bytwo}
Let $P=(a,b)$ be a $K$-point on $\CC$ and $D$ an effective divisor on $\CC$ such that $m=\deg(D)\le g$ and $\supp(D)$ does not contain $\infty$. Suppose that the degree zero divisor $2D+\iota(P)-(2m+1)(\infty)$ is principal. Then:
\begin{itemize}
\item[(i)]
$m=g$ and there exists a polynomial $v_D(x)\in K[x]$ such that $\deg(v)\le g$ and the divisor of $y-v_D(x)$ coincides with $2D+\iota(P)-(2g+1)(\infty)$. In particular, $-b=v(a)$.
\item[(ii)]
If a point $Q$ lies in  $\supp(D)$ then $\iota(Q)$ does not lie in  $\supp(D)$. In particular, 
\begin{enumerate}
\item
none of $\W_i$ lies in  $\supp(D)$;
\item
$D-g(\infty)$ is reduced.
\end{enumerate}
\item[(iii)]
The point $P$ does not lie in  $\supp(D)$.
\end{itemize}
\end{cor}

\begin{proof}
One has only to apply Lemma \ref{key} to the divisor $2D+\iota(P)$ of {\sl odd} degree $2m+1\le 2g+1$ and notice that $\iota(P)=(a,-b)$ is a zero of $y-v(x)$ while $\iota(\W_i)=\W_i$ for all $i=1, \dots , 2g+1$.
\end{proof}

Let $d \le g$ be a positive integer and $\Theta_d \subset J$ be the image of the regular map
$$\CC^{d} \to J, \ (Q_1, \dots , Q_{d}) \mapsto \sum_{i=1}^{d} Q_i\subset J.$$
It is well known that $\Theta_d$ is a closed $d$-dimensional subvariety of $J$ that coincides with $\CC$ for $d=1$ and with $J$ if $d \ge g$; in addition, $\Theta_d\subset\Theta_{d+1}$ for all $d$. Clearly, each $\Theta_d$ is stable under multiplication by $-1$ in $J$.
We write $\Theta$ for the $(g-1)$-dimensional {\sl theta divisor} $\Theta_{g-1}$.  

\begin{thm}
\label{notheta}
Suppose that $g>1$ and 
let $$\CC_{1/2}:=2^{-1}\CC \subset J$$
 be the preimage of $\CC$ with respect to multiplication by 2 in $J$.  Then the intersection of $\CC_{1/2}(K)$ and $\Theta$ 
consists of points of order dividing  $2$ on $J$. In particular, the intersection of $\CC$ and $C_{1/2}$ consists of $\infty$ and all $\W_i$'s. 
\end{thm}
\begin{proof}
Suppose that $m \le g-1$ is a positive integer and we have $m$  (not necessarily distinct) points $Q_1, \dots Q_m$  of $\CC(K)$ and  a point $P\in \CC(K)$ such that in $J(K)$
$$2\sum_{j=1}^m Q_j=P.$$
We need to prove that $P=\infty$, i.e.,  it is the zero of group law in $J$ and therefore
$\sum_{j=1}^m Q_j$ is an element of order 2 (or 1) in $J(K)$. Suppose that this is not true.  Decreasing $m$ if necessary, we may and will assume that none of $Q_j$ is $\infty$ (but $m$ is still positive and does not exceed $g-1$).  Let us consider the effective degree $m$ divisor $D=\sum_{j=1}^m (Q_j)$ on $\CC$. The equality in $J$ means that the divisors $2[D-m(\infty)]$ and $(P)-(\infty)$  on $\CC$ are linearly equivalent.  This means that the divisor $2D+(\iota(P))-(2m+1)(\infty)$ is {\sl principal}. Now Corollary \ref{bytwo} tells us that $m=g$, which is not the case. The obtained contradiction proves that the intersection of $\CC_{1/2}$ and $\Theta$ consists of points of order 2 and 1.  

Since $g>1$,  $\CC\subset \Theta$ and therefore the intersection of $\CC$ and $\CC_{1/2}$ also consists of points of order 2 or 1, i.e., lies in the union of $\infty$ and all $\W_i$'s.  Conversely, since each $\W_i$  has order  $2$ in $J(K)$ and $\infty$ has order 1, they all lie in $\CC_{1/2}$ (and, of course, in $\CC$).
\end{proof}

\begin{rem}
It is known \cite[Ch. VI, last paragraph of Sect. 11, p. 122]{Serre} that the curve $\CC_{1/2}$ is irreducible. (Its projectiveness and smoothness   follow readily from  the  projectiveness and smoothness of $\CC$ and the \'etaleness of multiplication by 2 in $J$.) See \cite{Flynn} for an explicit description of equations that cut out $\CC_{1/2}$ in a projective space.
\end{rem}

\begin{cor}
Suppose that $g>1$. Let $n$  an integer such that $3 \le n \le 2g$.
Then $\CC(K)$ does not contain a point of order $n$ in $J(K)$. In particular, $\CC(K)$  does not contain  points of order $3$ or $4$.
\end{cor}

\begin{proof}
Suppose that such a point say, $P$ exists.  Clearly, $P$ is neither $\infty$ nor one of $\W_i$, i.e., $P \ne \iota(P)$.

Suppose that $n$ is odd. Then we have $n=2m+1$ with $1\le m<g$. This implies that $mP \in \Theta$ and
$$2(mP)=2mP=-P=\iota(P) \in \CC(K).$$
It follows from Theorem \ref{notheta} that either $mP=0$ in $J(K)$ or
$(2m)P=2(mP)=0$ in $J(K)$. However, the order of $P$ in $J(K)$ is 
$n=2m+1>m\ge 1$ and we get a desired contradiction.

Assume now that $n$ is even. Then we have $n=2m$ with $1<m\le g$. Then
$mP$ has order $2$ in $J(K)$. It follows that 
$$mP=-mP=m(-P)=m\ \iota(P).$$
 This means that the  degree zero divisors
$m(P)-m(\infty)$ and $m(\iota(P))-m(\infty)$ belong to the same linear equivalence class. Since both divisors are {\sl reduced}, they must coincide (see \cite[Ch. 13, Prop. 13.6 on p. 413]{Wash}). This implies that $P=\iota(P)$, which is not the case and we get a desired contradiction.
\end{proof}

\begin{rem}
If $\fchar(K)=0$ and $g>1$ then the famous theorem of M. Raynaud (conjectured by Yu.I. Manin and D. Mumford) asserts that  an arbitrary genus $g$ smooth projective curve over $K$ embedded into its jacobian contains only finitely many torsion points \cite{Raynaud}.  Using a $p$-adic approach, B. Poonen \cite{PoonenEXP} developed and implemented  an algorithm that finds all complex torsion points on genus $2$ hyperelliptic curves $\CC:y^2=f(x)$ such that $f(x)$ has rational coefficients. (See also \cite{PoonenStoll}.)
\end{rem}

\begin{thm}
\label{ThetaD}
Suppose that $g>1$ and 
let  $N>1$ be positive integer.  Suppose that $N \le 2g-1$ and let us put
$$d(N)=\left[\frac{2g}{N+1}\right].$$
 Let $K_0$ be a subfield of $K$ such that $f(x)\in K_0[x]$.
Let $\a$ be a $K$-point on $\Theta_{d(N)}$. Suppose that there is a field automorphism $\sigma\in \Aut(K/K_0)$ such that $\sigma(\a)=N\a$ or $-N\a$.
 Then $\a$ has order 1 or 2 in $J(K)$. 
\end{thm}

\begin{proof}
Clearly, $(N+1)\cdot d(N)<2g+1$.
Let us assume that  $2\a\ne 0$ in $J(K)$. We need  to arrive to a contradiction.
Then there is a positive integer $r\le d(N)$ and a sequence of points $P_1, \dots, P_r$ of $\CC(K)\setminus{\infty}$ such that
 $\tilde{D}:=\sum_{j=1}^r(P_j)-r(\infty)$ is the Mumford's representation of $\a$ while (say) $P_1$ does {not} coincides with 
any of $W_i$ (here we use the assumption that $2\a \ne 0$); we may also assume that $P_1$ has the largest {\sl multiplicity} 
 say, $M$ among $\{P_1, \dots, P_r\}$. (In particular, none of $P_j$'s coincides with $\iota(P_1)$.)
Then $\sigma(\tilde{D})=\sum_{j=1}^r(\sigma(P_j))-r(\infty)$ is the Mumford's representation of $\sigma\a$. In particular,  the multiplicity of each $\sigma(P_j)$ in  $\sigma(\tilde{D})$ does {\sl not} exceed $M$; similarly, the multiplicity of each $\iota\sigma(P_j)$ in  $\iota\sigma(\tilde{D})$ does {\sl not} exceed $M$.

Suppose that $\sigma(\a)=N\a$.
$$N\tilde{D}+\iota\sigma(\tilde{D}) =N\left[\sum_{j=1}^r(P_j)\right]+\left[\sum_{j=1}^r(\iota\sigma(P_j))\right]-r(N+1)(\infty)$$
is a principal divisor on $\CC$. Since $m:=r(N+1) \le (N+1)\cdot d(N)<2g+1$,
 we are in position to apply Lemma \ref{key}, which tells us right away that $m$ is {\sl even} 
and there is a monic polynomial $u(x)$ of degree $m/2$, whose divisor coincides with $N\tilde{D}+\iota\sigma(\tilde{D})$. This implies that  a point $Q\in \CC(K)$ appears in $N\tilde{D}+\iota\sigma(\tilde{D})$ with the same multiplicity as $\iota{Q}$. It follows that $\iota{P_1}$ is (at least) one of $\iota\sigma(P_j)$'s.
Clearly, the multiplicity of $P_1$ in $N\tilde{D}+\iota\sigma(\tilde{D})$ is, at least, $NM$
 while the multiplicity of $\iota(P_1)$ is, at most, $M$. This implies that $NM\le M$. Taking into account that $N>1$,
 we obtain the desired contradiction. 

If $\sigma(\a)=-N\a$ then literally the same arguments applied to to the principal divisor
$$N\tilde{D}+\sigma(\tilde{D}) =N\left[\sum_{j=1}^r(P_j)\right]+\left[\sum_{j=1}^r(\sigma(P_j))\right]-r(N+1)(\infty)$$
also lead to the contradiction.
\end{proof}

\section{Division by 2}
\label{division}
Suppose we are given a point
$$P=(a,b) \in \CC(K) \subset J(K).$$
Since $\dim(J)=g$, 
there are exactly $2^{2g}$ points $\a \in J(K)$ such that
$$P=2\a \in J(K).$$
Let us choose such an $\a$.
Then there is exactly one effective divisor 
$$D=D(\a)\eqno(1)$$
 of positive degree $m$ on $\CC$ such that  $\supp(D)$ does {\sl not} contain $\infty$, the divisor $D-m(\infty)$ is reduced, and
 $$m \le g, \ \cl(D-m (\infty))=\a.$$ It follows that the divisor $2D+(\iota(P))-(2m+1)(\infty)$ is {\sl principal} and, thanks to  Corollary \ref{bytwo},  $m=g$ and $\supp(D)$ does {\sl not} contains 
  any of $\W_i$. (In addition, $D-g(\infty)$ is reduced.) 
Then the degree $g$ effective divisor 
 $$D=D(\a)=\sum_{j=1}^{g}(Q_j)\eqno(2)$$
    with
$Q_i=(c_j,d_j)\in \CC(K)$.  Since none of $Q_j$ coincides with any of $\W_i$,
$$c_j \ne \alpha_i \ \forall i,j.$$
By  Corollary \ref{bytwo}, there is a polynomial $v_D(x)$ of degree $ \le g$ such that the degree zero divisor $$2D+(\iota(P))-(2g+1)(\infty)$$ is the divisor of $y-v_D(x)$. Since the points $\iota(P)=(a,-b)$  and all $Q_j$'s  are zeros of $y-v_D(x)$,
$$b=-v_D(a), \ d_j=v_D(c_j) \ \forall j=1, \dots , g.$$
It follows from Proposition 13.2 on pp. 409--410 of \cite{Wash} that 
$$\prod_{i=1}^{2g+1}(x-\alpha_i)-v_D(x)^2=f(x)-v_D(x)^2=(x-a)\prod_{j=1}^g (x-c_j)^2.\eqno(3)$$
In particular, $f(x)-v_D(x)^2$ is divisible by
$$u_D(x):=\prod_{j=1}^g (x-c_j).\eqno(4)$$

\begin{rem}
Summing up:
$$D=D(\a)=\sum_{j=1}^{g}(Q_j), \ Q_j=(c_j,v_D(c_j)) \ \forall j=1, \dots , g$$
and the dgree $g$ monic polynomial $u_D(x)=\prod_{j=1}^g (x-c_j)$ divides $f(x)-v_D(x)^2$.  By Prop. 13.4 on p. 412 of \cite{Wash}, this implies that reduced $D-g(\infty)$ coincides with the gcd of the divisors of $u_D(x)$ and  $y-v_D(x)$. 
Therefore
the  pair $(u_D,v_D)$ is the Mumford's representation of $\a$
if $$\deg(v_D)<g=\deg(u_D).$$ This is not always the case: it may happen that $\deg(v_D)=g=\deg(u_D)$ (see below). However, if we replace $v_D(x)$ by its remainder with respect to the division by $u_D(x)$ then we get the Mumford's representation  of $\a$ (see below).
\end{rem}

If in (3) we put $x=\alpha_i$ then we get
$$-v_D(\alpha_i)^2=(\alpha_i-a)\left(\prod_{j=1}^g(\alpha_i-c_j)\right)^2,$$
i.e.,
$$v_D(\alpha_i)^2=(a-\alpha_i)\left(\prod_{j=1}^g(c_j-\alpha_i)\right)^2 \ \forall \ i=1, \dots , 2g+1.$$
Since none of $c_j-\alpha_i$ vanishes, we may define
$$r_i=r_{i,D}:=\frac{v_D(\alpha_i)}{\prod_{j=1}^g(c_j-\alpha_i)} \eqno(5)$$
with
$$r_i^2=a-\alpha_i \ \forall \ i=1, \dots , 2g+1 \eqno(6)$$
and
$$\alpha_i= a-r_i^2, \  c_j-\alpha_i=r_i^2-a+c_j \ \forall \ i=1, \dots , 2g+1; j=1, \dots , g.$$
Clearly, all $r_i$'s are {\sl distinct} elements of $K$, because their squares are obviously distinct. (By the same token, $r_{j_1} \ne \pm r_{j_2}$ if $j_1\ne j_2$. Notice that
$$\prod_{i=1}^{2g+1}r_i= \pm b, \eqno(7)$$
because
$$b^2=\prod_{i=1}^{2g+1}(a-\alpha_i)=\prod_{i=1}^{2g+1} r_i^2.)\eqno(8)$$
Now we get
$$r_i=\frac{v_D(a-r_i^2)}{\prod_{j=1}^g(r_i^2-a+c_j)},$$
i.e.,
$$r_i \prod_{j=1}^g(r_i^2-a+c_j)-v_D(a-r_i^2)=0 \ \forall \ i=1, \dots 2g+1.$$
This means that the degree $(2g+1)$ {\sl monic} polynomial (recall that $\deg(v_D)\le g$)
$$h_{\mathbf{r}}(t):=t \prod_{j=1}^g (t^2-a+c_j) -v(a-t^2)$$
has $(2g+1)$ {\sl distinct} roots $r_1, \dots, r_{2g+1}$. This means that
$$h_{\mathbf{r}}(t)= \prod_{i=1}^{2g+1}(t-r_i).$$
Clearly, $t \prod_{j=1}^g (t^2-a+c_i)$ coincides with the {\sl odd part} of $h_{\mathbf{r}}(t)$ while $-v_D(a-t^2)$ coincides with the {\sl even part} of $h_{\mathbf{r}}(t)$. In particular, if we put $t=0$ then we get
$$ (-1)^{2g+1}\prod_{i=1}^{2g+1}r_i=-v_D(a)=b,$$
i.e.,
$$\prod_{i=1}^{2g+1}r_i=- b.\eqno(9)$$
Let us define 
$$\mathbf{r}=\mathbf{r}_D:=(r_1, \dots , r_{2g+1}) \in K^{2g+1}.$$
Since $$\mathbf{s}_i(\mathbf{r})=\mathbf{s}_i(r_1, \dots , r_{2g+1})$$  is the $i$th basic symmetric function in 
$r_1, \dots , r_{2g+1}$, 
$$h_{\mathbf{r}}(t)=t^{2g+1}+\sum_{i=1}^{2g+1} (-1)^{i}\mathbf{s}_i(\mathbf{r}) t^{2g+1-i}=[t^{2g+1}+\sum_{i=1}^{2g} (-1)^{i}s_i(\mathbf{r}) t^{2g+1-i}] +b.$$
Then
$$t\prod_{j=1}^g (t^2-a+c_j)=t^{2g+1}+\sum_{j=1}^g \mathbf{s}_{2j}(\mathbf{r})t^{2g+1-2j},$$
$$-v_D(a-t^2)=[-\sum_{j=1}^{g} \mathbf{s}_{2j-1}(\mathbf{r}) t^{2g-2j+2}]+b.$$
It follows that
$$\prod_{j=1}^g (t-a+c_j)=t^g+\sum_{j=1}^g \mathbf{s}_{2j-1}(\mathbf{r})t^{g-j},$$
$$v_D(a-t)=\sum_{j=1}^{g} \mathbf{s}_{2j-1}(\mathbf{r}) t^{g-j+1}-b.$$
This implies that 
$$v_D(t)=\left[\sum_{j=1}^{g} \mathbf{s}_{2j-1}(\mathbf{r}) (a-t)^{g-j+1}\right]-b.\eqno(10)$$
It is also clear that if we consider the degree $g$ monic polynomial
$$U_{\mathbf{r}}(t):=u_D(t)=\prod_{j=1}^g (t-c_j)$$
then
$$
U_{\mathbf{r}}(t)=(-1)^g \left[(a-t)^g+\sum_{j=1}^g \mathbf{s}_{2j}(\mathbf{r})(a-t)^{g-j}\right]. \eqno(11)
$$

 Recall that $\deg(v_D) \le g$ and notice that the coefficient of $v(x)$ at $x^g$ is $(-1)^{g}\mathbf{s}_1(\mathbf{r})$. This implies that  the polynomial
$$V_{\mathbf{r}}(t):=v_D(t)-(-1)^{g}\mathbf{s}_1(\mathbf{r}) U_{\mathbf{r}}(t)=$$
$$\left[\sum_{j=1}^{g} \mathbf{s}_{2j-1}(\mathbf{r}) (a-t)^{g-j+1}\right]-b
-
\mathbf{s}_1(\mathbf{r})\left[(a-t)^g+\sum_{j=1}^g \mathbf{s}_{2j}(\mathbf{r})(a-t)^{g-j}\right]
\eqno(12)
$$
has degree $<g$, i.e.,  
$$\deg(V_{\mathbf{r}})<\deg(U_{\mathbf{r}})=g.$$
 Clearly, $f(x) - V_{\mathbf{r}}(x)^2$ is still divisible by $U_{\mathbf{r}}(x)$, because $u_D(x)=U_{\mathbf{r}}(x)$ divides both
$f(x)-v_D(x)^2$ and $v_D(x)- V_{\mathbf{r}}(x)$.
On the other hand,
$$d_j=v_D(c_j)=V_{\mathbf{r}}(c_j) \ \forall j=1, \dots g,$$
because $U_{\mathbf{r}}(x)$ divides  $v_D(x)- V_{\mathbf{r}}(x)$ and vanishes at all $b_j$. Actually, $\{b_1, \dots, b_g\}$ is the list of all roots (with multiplicities) of $U_{\mathbf{r}}(x)$. So,
$$D=D(\a)=\sum_{j=1}^{g}(Q_j), \ Q_j=(c_j,v_D(c_j))=(c_j,V_{\mathbf{r}}(c_j))  \ \forall j=1, \dots , g.$$
This implies (again via Prop. 13.4 on p. 412 of \cite{Wash}) that reduced $D-g (\infty)$ coincides with the gcd of the divisors of $U_{\mathbf{r}}(x)$ and  $y-V_{\mathbf{r}}(x)$. It follows that  the pair  $(U_{\mathbf{r}}(x), V_{\mathbf{r}}(x))$ is the {\sl Mumford's representation}  of  $\cl(D-g (\infty))=\a$. 
So,  the formulas (11) and (12) give us an explicit construction of ($D(\a)$ and) $\a$ in terms of $\mathbf{r}=(r_1, \dots, , r_{2g+1})$ for each of $2^{2g}$ choices of $\a$ with $2\a=P\in J(K)$.  
On the other hand,  in light of (6)-(8), there is exactly the same number $2^{2g}$ of  choices of  square roots $\sqrt{a-\alpha_i}$ ($1\le i \le 2g$), whose product is $-b$. Combining it with (9), we obtain  that for each choice
of square roots $\sqrt{a-\alpha_i}$'s  with $\prod_{i=1}^{2g+1}\sqrt{a-\alpha_i}=-b$ there is precisely one $\a \in J(K)$ with $2\a=P$ such that the corresponding $r_i$ defined by (5) coincides with  chosen $\sqrt{a-\alpha_i}$ for all $i=1, \dots , 2g+1$, and the Mumford's  representation $(U_{\mathbf{r}}(x), V_{\mathbf{r}}(x))$  for this $\a$ is given by explicit formulas (11)-(12). This gives us the following assertion.

\begin{thm}
\label{main}
Let $P=(a,b)\in \CC(K)$. Then the $2^{2g}$-element set 
$$M_{1/2,P}:=\{\a \in J(K)\mid 2\a=P\in \CC(K)\subset J(K)\}$$
can be described as follows.  Let $\RR_{1/2,P}$ be the set of all $(2g+1)$-tuples $\rr=(\rr_1, \dots , \rr_{2g+1})$ of elements of $K$ such that
$$\rr_i^2=a-\alpha_i \ \forall \ i=1, \dots , 2g+1;  \ \prod_{i=1}^{2g+1}\rr_i=-b.$$ 
Let $\mathbf{s}_i(\rr)$ be the $i$th basic symmetric function in  $\rr_1, \dots , \rr_{2g+1}$. Let us put
$$
U_{\rr}(x)=(-1)^g \left[(a-x)^g+\sum_{j=1}^g \mathbf{s}_{2j}(\rr)(a-x)^{g-j}\right],$$
$$V_{\rr}(x)=\left[\sum_{j=1}^{g} \mathbf{s}_{2j-1}(\rr) (a-x)^{g-j+1}\right]-b
-
\mathbf{s}_1(\rr)\left[(a-x)^g+\sum_{j=1}^g \mathbf{s}_{2j}(\rr)(a-x)^{g-j}\right].$$
Then there is a natural bijection between  $\RR_{1/2,P}$ and $M_{1/2,P}$ such that $\rr \in  \RR_{1/2,P}$  corresponds to $\a_{\rr}\in M_{1/2,P}$ with Mumford's representation  $(U_{\rr},V_{\rr})$.  More explicitly, if $\{c_1, \dots, c_g\}$ is the list of $g$ roots (with multiplicities) of $U_{\rr}(x)$ then $\rr$ corresponds to
$$\a_{\rr}=\cl(D-g(\infty))\in J(K), \ 2\a_{\rr}=P$$
where the divisor
$$D=D(\a_{\rr})=\sum_{j=1}^g (Q_j), \ Q_j=(b_j,V_{\rr}(b_j))\in \CC(K) \ \forall \ j=1, \dots, g.$$
In addition, none of $\alpha_i$ is a root of $U_{\rr}(x)$ (i.e., the polynomials $U_{\rr}(x)$ and $f(x)$ are relatively prime) and
$$\rr_i=\mathbf{s}_1(\rr)+(-1)^g\frac{V_{\rr}(\alpha_i)}{U_{\rr}(\alpha_i)}
 \ \forall \ i=1, \dots , 2g+1.$$
\end{thm}

\begin{proof}
Actually we have already proven all the assertions of Theorem \ref{main} except the last formula for $\rr_i$. It follows from (4) and (5) that
$$\rr_i=(-1)^g \frac{v_{D(\a_{\rr})}(\alpha_i)}{u_{D(\a_{\rr})}(\alpha_i)}=
(-1)^g \frac{v_{D(\a_{\rr})}(\alpha_i)}{U_{\rr}(\alpha_i)}.$$
It follows from (12) that
$$v_{D(\a_{\rr})}(x)=(-1)^{g}\mathbf{s}_1(\rr) U_{\rr}(x)+V_{\rr}(x).$$
This implies that
$$\rr_i=(-1)^g\frac{(-1)^{g}\mathbf{s}_1(\rr) U_{\rr}(\alpha_i)+V_{\rr}(\alpha_i)}{U_{\rr}(\alpha_i)}=\mathbf{s}_1(\rr)+(-1)^g\frac{V_{\rr}(\alpha_i)}{U_{\rr}(\alpha_i)}.$$
\end{proof}

\begin{ex}
Let us take as $P=(a,b)$ the point $\W_{2g+1}=(\alpha_{2g+1},0)$.  Then $b=0$ and  $\rr_{2g+1}=0$. We have $2g$ arbitrary independent choices of (nonzero) square roots $\rr_j=\sqrt{\alpha_{2g+1}-\alpha_j}$ with $1 \le j \le 2g$ (and always get an element of $\RR_{1/2,P}$).  Now Theorem \ref{main} gives us (if we put $a=\alpha_{2j+1},b=0$)  all $2^{2g}$  points $\a_{\rr}$ of order 4 in $J(K)$ with $2\a_{\rr}=\W_{2j+1}$.
\end{ex}

 \section{Rationality Questions}
 \label{rat}
 Let $K_0$ be a subfield of $K$ and $K_0^{\sep}$ its separable algebraic closure in $K$.   Recall that $K_0^{\sep}$  is separably closed.
 Clearly,
 $$\fchar(K_0)=\fchar(K_0^{\sep})=\fchar(K) \ne 2.$$
  Let us assume that $f(x)\in K_0[x]$, i.e., all the coefficients of $f(x)$ lie in $K_0$. However, we don't make any additional assumptions about its roots $\alpha_j$; still, all of them lie in $K_0^{\sep}$, because $f(x)$ has no multiple roots. Recall that both $\CC$ and $J$ are defined over $K_0$; the point 
 $\infty \in \CC(K_0)$ and therefore the embedding $\CC\hookrightarrow J$ is defined over $K_0$; in particular, $\CC$ is a closed algebraic $K_0$-subvariety of $J$.
 
 Let  us assume that our $K$-point $P=(a,b)$ of $\CC$  lies in $\CC(K_0^{\sep})$, i.e., $a,b\in K_0^{\sep}$ and
 $$P=(a,b)\in \CC(K_0^{\sep})\subset J(K_0^{\sep})\subset J(K).$$
 In the notation of Theorem \ref{main},  for each $\rr\in M_{1/2,P}$
 all its components $\rr_i$ lie in $K_0^{\sep}$, because $\rr_i^2=a-\alpha_i \in K_0^{\sep}$. This implies that 
  the monic degree $2g+1$ polynomial
 $$h_{\rr}(t)=\prod_{i=1}^{2g+1}(t-\rr_i)=t^{2g+1}+\sum_{i=1}^{2g}(-1)^i \mathbf{s}_i(\rr)t^{2g+1-i} \in K_0^{\sep}[t],$$
 i.e., all  $\mathbf{s}_i(\rr) \in K_0^{\sep}$.
 It follows immediately from the explicit formulas above that the Mumford representation $(U_{\rr},V_{\rr})$ of $\a_r=\cl(D(\a_r)-g(\infty))$ consists of polynomials  $U_{\rr}$ and $V_{\rr}$ with coefficients in $K_0^{\sep}$. In addition, $\a_{\rr}$ lies in $J(K_0^{\sep})$, because $2\a_{\rr}=P \in J(K_0^{\sep})$, the multiplication by $2$ in $J$ is an \'etale map and $K_0^{\sep}$ is separably closed.

\begin{lem}
\label{sep}
 Suppose that either $K_0$ is a perfect field (e.g., $\fchar(K)=0$ or $K_0$ is finite) or $\fchar(K_0)>g$.  Suppose that 
 $$P=(a,b) \in \CC(K_0^{\sep})\subset J(K_0^{\sep}).$$
 Then for all  $\rr\in \RR_{1/2,P}$ the Mumford representation $(U_{\rr},V_{\rr})$ of $\a_r=\cl(D(\a_r)-g(\infty))$ enjoys the following properties.
 
 \begin{itemize}
 \item[(i)]
The polynomial  $U_{\rr}(x)$ splits over $K_0^{\sep}$, i.e., all its roots $b_j$  lie  in $K_0^{\sep}$.
\item[(ii)]
 The  divisor
 $$D=D(\a_{\rr})=\sum_{j=1}(Q_j)$$
 where
 $$Q_j=(c_j,V_{\rr}(c_j))\in \CC(K_0^{\sep}) \ \forall \ j=1, \dots , g.$$
 \end{itemize}
\end{lem}

\begin{proof}
If $K_0$ is perfect then $K_0^{\sep}$ is algebraically closed and there is nothing to prove. So, we may assume that $\fchar(K_0^{\sep}) =\fchar(K_0)>g$.
In order to prove (i), recall that $\deg(U_{\rr})=g$. Every root $c_j$ of $U_{\rr}(x)$ lies in $K$ and the algebraic  field extension $K_0^{\sep}(b_j)/K_0^{\sep}$ has finite degree that does not exceed
$$\deg(U_{\rr})=g<\fchar(K_0^{\sep})$$
and therefore this degree is {\sl not} divisible by $\fchar(K_0^{\sep})$. This implies that the field extension $K_0^{\sep}(b_j)/K_0^{\sep}$ is separable. Since $K_0^{\sep}$ is separably closed, the overfield $K_0^{\sep}(c_j)=K_0^{\sep}$, i.e., 
  $c_j$ lies  in $K_0^{\sep}$. This proves (i). As for (ii), since   $V_{\rr}(x)\in K_0^{\sep}[x]$ and all $c_j\in K_0^{\sep}$, we have $V_{\rr}(c_j)\in K_0^{\sep}$ and therefore $Q_j=(c_j,V_{\rr}(c_j))\in \CC(K_0^{\sep})$. This proves (ii).
\end{proof}

\begin{rem}
If $g=2$ then the conditions of Lemma \ref{sep} do not impose any additional restrictions on $K_0$. (The case $\fchar(K)=2$ was excluded from the very beginning.)
\end{rem}
 
 \begin{rem}
If $P=(a,b)\in \CC(K_0)$ then for each  $\rr\in \RR_{1/2,P}$
$$\mathbf{s}_{2g+1}(\rr)=(-1)^{2g+1}\prod_{i=1}^{2g+1}\rr_i=-\prod_{i=1}^{2g+1}\rr_i=-(-b)=b\in K_0.$$
This observation (reminder) explains the omission of $i=2g+1$ in the following statement.
\end{rem}

 \begin{thm}
 \label{mainR}
 Suppose that a point
 $$P =(a,b) \in \CC(K_0)\subset J(K_0),$$
 i.e.,
 $$a,b \in K_0, \ b^2=f(a).$$
 If $\rr$ is an element of $\RR_{1/2,P}$ then $\a_{\rr}$ lies in $J(K_0)$ if and only if
 $h_{\rr}(t)$ lies in $K_0[t]$, i.e.,
 $$\mathbf{s}_i(\rr) \in K_0 \ \forall \ i=1, \dots , 2g.$$
\end{thm}

\begin{proof}
Let $\bar{K}_0$ be the algebraic closure of $K_0$. Clearly, $\bar{K}_0$ is algebraically closed and
$$K_0\subset K_0^{\sep}\subset \bar{K}_0\subset K.$$
In the course of the proof we may and will assume that $K=\bar{K}_0$.

Let $\rr$ be an element of$\RR_{1/2,P}$. We know that 
$\a_{\rr} \in J(K_0^{\sep})$ and the corresponding polynomials $U_{\rr}(x)$ and $V_{\rr}(x)$ have coefficients in $K_0^{\sep}$. This means that there is a finite Galois field extension $E/K_0$ with Galois group $\Gal(E/K)$ such that 
$$K_0\subset E \subset K_0^{\sep}$$
such that
$$\a_{rr}\in J(E); \ U_{\rr}(x), V_{\rr}(x) \in E[x].$$
Let $\Aut(K/K_0)$ be the group of all field automorphisms of $K$ that leave invariant every element of $K_0$. Clearly, the (sub)field $E$  is $\Aut(K/K_0)$-stable and the natural (restriction) group homomorphism
$$\Aut(K/K_0) \to \Gal(E/K_0)$$
is {\sl surjective}. Since the subfield $E^{\Gal(E/K_0)}$ of Galois invariants coincides with $K_0$, we conclude that the subfield of invariants $E^{\Aut(K/K_0)}$ also coincides with $K_0$. It follows that
$$E[x]^{\Aut(K/K_0)}=K_0[x], \ J(E)^{\Aut(K/K_0)}=J(K_0).$$
 Clearly, for each $\sigma \in \Aut(K/K_0)$ the Mumford representation of $\sigma \a_{\rr}$ is $(\sigma U_{\rr}, \sigma V_{\rr})$.

 Now let us assume that  $\a_{\rr}\in J(K_0)$. Then 
 $$\sigma \a_{\rr}=\a_{\rr} \ \forall \ \sigma \in \Aut(K/K_0).$$
 The uniqueness of Mumford's representations implies that
 $$\sigma U_{\rr}(x)=U_{\rr}(x), \ \sigma V_{\rr}(x)=V_{\rr}(x) \ \forall \ \sigma \in \Aut(K/K_0).$$
 It follows  that
 $$U_{\rr}(x), V_{\rr}(x) \in K_0[x].$$
 Taking into account that $a,b \in K_0$, we obtain from the formulas in Theorem \ref{main} that 
$$\mathbf{s}_i(\rr) \in K_0 \ \forall \ i=1, \dots , 2g.$$

Conversely, let us assume that for a certain $\rr \in \RR_{1/2,P}$
$$\mathbf{s}_i(\rr) \in K_0 \ \forall \ i=1, \dots , 2g.$$
(We know that $\mathbf{s}_{2g+1}(\rr)$ also lies in $K_0$.) This implies that
both $U_{\rr}(x)$ and $V_{\rr}(x)$ lie in $K_0[x]$. In other words,
 $$\sigma U_{\rr}(x)=U_{\rr}(x), \ \sigma V_{\rr}(x)=V_{\rr}(x) \ \forall \ \sigma \in \Aut(K/K_0).$$
 This means that for every $\sigma \in \Aut(K/K_0)$ both $\a_{\rr}$ and $\sigma \a_{\rr}$ have the  same Mumford representation, namely, $(U_{\rr}, V_{\rr})$. This implies that 
 $$\sigma \a_{\rr}=\a_{\rr}  \ \forall \ \sigma \in \Aut(K/K_0),$$
 i.e.,
 $$\a_{\rr} \in J(E)^{\Aut(K/K_0)}=J(K_0).$$
\end{proof}

\begin{thm}
\label{full}
Suppose that a point
 $$P =(a,b) \in \CC(K_0)\subset J(K_0),$$
 i.e.,
 $$a,b \in K_0, \ b^2=f(a).$$  Then the following conditions are equivalent.
 
 \begin{itemize}
 \item[(i)]
 $\alpha_i \in K_0$ and $a-\alpha_i$  is a square in $K_0$
  for all $i$ with $1\le i \le 2g+1$.
 \item[(ii)]
 All $2^{2g}$ elements $\a\in J(K)$ with $2\a=P$ actually lie in $J(K_0)$.
 \end{itemize}
\end{thm}

\begin{proof}
Assume (i). Then $\a=\a_{\rr}$ for a certain $\rr \in \RR_{1/2,P}$. Our assumptions imply that all $\rr_i=\sqrt{a-\alpha_i}$ lie in $K_0$ and therefore
$$\mathbf{s}_i(\rr) \in K_0 \ \forall \ i=1, \dots , 2g.$$
Now Theorem \ref{mainR} tells us that 
$\a_{\rr}\in J(K_0)$.
This proves (ii).

Assume (ii). It follows from Theorem \ref{mainR} that $\mathbf{s}_i(\rr)\in K_0$ for all $\rr \in \RR_{1/2,P}$ and $i$ with $1 \le i \le 2g+1$. In particular, for $i=1$
$$\sum_{i=1}^{2g+1}\rr_i=\mathbf{s}_1(\rr)\in K_0 \ \forall \ \rr\in  \RR_{1/2,P}.$$
Pick any $\rr\in \RR_{1/2,P}$ and for {\sl any} index $l$ ($1\le l\le 2j+1$) consider
$\rr^{(l)} \in \RR_{1/2,P}$ such that
$$\rr^{(l)}_l=\rr_l, \ \rr^{(l)}_i =-\rr_i \ \forall i \ne l.$$
We have
$$\mathbf{s}_1(\rr)\in K_0,  \ -2\mathbf{s}_1(\rr)+2 \rr_l=\mathbf{s}_1(\rr^{(l)})\in K_0.$$
This implies that $\rr_l \in K_0$. Since $\rr_{\ell}^2=a-\alpha_l$ and $a \in K_0$, we conclude that $\alpha_l$ lies in $K_0$ and $a-\alpha_l$ is a square in $K_0$. This proves (i).
\end{proof}

\begin{rem}
In the case of elliptic curves (i/e., when $g=1$) Theorem \ref{full} is well known, see, e.g., \cite[p. 269--270]{Cassels}.

\end{rem}

The following assertion was inspired by results of Schaefer \cite{Schaefer}.

\begin{thm}
\label{mainS}
Let us consider the $(2g+1)$-dimensional commutative semisimple $K_0$-algebra
$L=K_0[x]/f(x)K_0[x]$.  

A $K_0$-point $P=(a,b)$ on $\CC$ is divisible by $2$ in $J(K_0)$ if and only if
$$(a-x)+f(x)K_0[x] \in K_0[x]/f(x)K_0[x]=L$$
is a square in $L$.
\end{thm}

\begin{proof}
 For each 
$q(x)\in K_0[x]$ we write $\overline{q(x)}$ for its image in $K_0[x]/f(x)K_0[x]$.

For each $i=1, \dots 2g+1$ there is a homomorphism of $K_0$-algebras
$$\phi_i: L=K_0[x]/f(x)K_0[x] \to K_0^{\sep}, \ \overline{q(x)}=q(x)+f(x)K_0[x]\mapsto q(\alpha_i);$$
 the intersection of the kernels of all $\phi_i$ is $\{0\}$. Indeed, if $\overline{q(x)}\in \ker(\phi_i)$ then $q(x)$ is divisible by $x-\alpha_i$ and therefore if $\overline{q(x)}$ lies in $\ker(\phi_i)$ for all $i$  then $q(x)$ is divisible by $\prod_{i=1}^{2g+1}(x-\alpha_i)=f(x)$, i.e., 
$\overline{q(x)}=0$ in $ K_0[x]/f(x)K_0[x]$.
Clearly,
$$\phi_i(\bar{x})=\alpha_i, \ \phi_i(\overline{a-x})=a-\alpha_i .$$
Since $f(x)$ lies in $K_0[x]$, the set of its roots  $\{\alpha_1, \dots, \alpha_{2g+1}\}$ is a Galois-stable subset of $K_0^{\sep}$. This implies that
 for each $q(x)\in K_0[x]$ and 
  $$\mathcal{Z}=\overline{q(x)}\in K_0[x]/f(x)K_0[x]$$ the product
$$\mathrm{H}_{\mathcal{Z}}(t)=H_{q(x)}(t) :=\prod_{i=1}^{2g+1}\left(t-\phi_i(\overline{q(x)}\right)=
\prod_{i=1}^{2g+1}(t-q(\alpha_i))$$
is a degree $(2g+1)$ monic polynomial with coefficients in $K_0$. In particular, if $q(x)=a-x$ then
$$H_{a-x}(t)=\mathrm{H}_{\overline{a-x}}(t) =\prod_{i=1}^{2g+1}[t-(a-\alpha_i)].$$

Assume that $P$ is divisible by $2$ in $J(K_0)$, i.e., there is $\a \in J(K_0)$ with $2\a=P$. It follows from Theorems \ref{main} and \ref{mainR} that there is $\rr \in \RR_{1/2,P}$ such that $\a_{\rr}=\a$ and all $\mathbf{s}_i(\rr)$  lie in $K_0$. This implies that both polynomials $U_{\rr}(x)$ and  $V_{\rr}(x)$ have coefficients in $K_0[x]$. Recall (Theorem \ref{main}) that $f(x)$ and $U_{\rr}(x)$ are relatively prime. This means that
$\overline{U_{\rr}(x)}=U_{\rr}(\bar{x})$ is a {\sl unit} in $K_0[x]/f(x)K_0[x]$. Therefore we may define
$$\mathcal{R}=\mathbf{s}_1(\rr)+(-1)^g\frac{V_{\rr}(\bar{x})}{U_{\rr}(\bar{x})} \in  K_0[x]/f(x)K_0[x].$$
The last formula of Theorem \ref{main} implies that  for all $i$ we have
$\phi_i(\mathcal{R})=\rr_i$ and therefore 
$$\phi_i(\mathcal{R}^2)=\rr_i^2=a-\alpha_i=\phi_i(a-\bar{x}).$$
This implies that $\mathcal{R}^2=a-\bar{x}$. It follows that 
$$a-\bar{x}=(a -x)+f(t)K_0[t] \in K_0[x]/f(x)K_0[x]$$
 is a {\sl square} in $K_0[x]/f(x)K_0[x]$.

Conversely, assume now that there is an element $\mathcal{R} \in L$ such that $$\mathcal{R}^2=a-\bar{x}=\overline{a-x}.$$
 This implies that
$$\phi_i(\mathcal{R})^2=\phi_i(\overline{a-x})=a-\alpha_i,$$
i.e.,
$$\phi_i(\mathcal{R})=\sqrt{a-\alpha_i} \ \forall \ i=1, \dots , 2g+1.$$
This implies that 
$$\prod_{i=1}^{2g+1}\phi_i(\mathcal{R})=\sqrt{f(a)}=\pm b.$$
 Since $(-1)^{2g+1}=-1$, replacing if necessary, $\mathcal{R}$ by $-\mathcal{R}$, we may and will assume that
$$\prod_{i=1}^{2g+1}\phi_i(\mathcal{R})=-b.$$
Now if we put 
$$\rr_i=\phi_i(\mathcal{R}) \ \forall \ i=1, \dots, 2g+1; \ \rr=(\rr_1, \dots, \rr_{2g+1})$$ 
then $\rr\in \RR_{1/2,P}$ and
$$h_{\rr}(t)=\prod_{i=1}^{2g+1}(t-\rr_i)=\prod_{i=1}^{2g+1}(t-\phi_i(\mathcal{R}))=\mathrm{H}_{\mathcal{R}}(t).$$
Since $\mathfrak{H}_{\mathcal{R}}(t)$ lies in $K_0[t]$, the polynomial $h_{\rr}(t)$ also lies in $K_0[t]$. It follows from Theorem \ref{mainR} that $\a_{\rr}\in J(K_0)$. Since $2\a_{\rr}=P$, the point $P$ is divisible by 2 in $J(K_0)$.
\end{proof}

\begin{rem} 
If one assumes additionally that $\fchar(K_0)=0$ and $P$ is none of $\W_i$ (i.e., $a \ne \alpha_i$ for any $i$) then the assertion of
 Theorem \ref{mainS} follows from  \cite[Th. 1.2 and the first paragraph of p. 224]{Schaefer}. 

\end{rem}

\section{Torsion points on $\Theta_d$}
\label{torsion}

We keep the notation of Section \ref{rat}. In particular, 
 $K_0$ be a subfield of $K$ such that 
$$f(x) \in K_0[x].$$
Notice  that  the involution $\iota$ is also defined over $K_0$,
 the absolute Galois group $\Gal(K_0)$ leaves invariant $\infty$ and permutes points of $\CC(K_0^{\sep})$; in addition, it permutes elements of $J(K_0^{\sep})$, respecting the group structure on $J(K_0^{\sep})$.

If $n$ is a positive integer that is {\sl not}
divisible by $\fchar(K)$ then we write $J[n]$ for the kernel of multiplication by $n$ in $J(K)$. It is well known that $J[n]$ is a free $Z/n\Z$-module of rank $2g$ that  lies in $J(K_0^{\sep})$; in addition, it is a $\Gal(K_0)$-stable subgroup of $J(K_0^{\sep})$, which gives us the (continuous) group homomorphism
$$\rho_{n,,J}:\Gal(K_0)\to \Aut_{\Z/n\Z}(J[n])$$
that defines the Galois action on $J[n]$. We write $\tilde{G}_{n,J,K_0}$ for the image
$$\rho_{n,J}(\Gal(K_0))\subset \Aut_{\Z/n\Z}(J[n]).$$
Let $\II_n$ be the identity automorphism of $J[n]$.
The following assertion was inspired by a work of F. Bogomolov \cite{Bogomolov} (where  the $\ell$-primary part of the Manin-Mumford conjecture was proven).

\begin{thm}
\label{tor}
Suppose that $g>1$ and $n\ge 3$ is an  integer that is not divisible by $\fchar(K)$. Let $N>1$ be an integer that is relatively prime to $n$ and such 
 that $N \le 2g-1$ and $\tilde{G}_{n,J,K_0}$  contains
either $N\cdot \II_n$ or $-N\cdot \II_n$.  Let us put $d(N):=[2g/(N+1)]$.

Then $\Theta_{d(N)}(K)$ does not contain nonzero points of order dividing  $n$ except points of order $1$ or $2$. In particular, if $n$ is odd then $\Theta_{d(N)}(K)$ does not contain nonzero points of order dividing  $n$.
\end{thm}
\begin{proof}
Clearly,  $(N+1)\cdot d(N)<2g+1$.
Suppose that $\b$ is a nonzero  point of order dividing $n$ in $\Theta_{d(N)}(K)$.  We need to prove that $2\b=0$.

Indeed, 
$\b \in J[n]\subset J(K_0^{\sep})$ and therefore
$$\b \in \Theta_d(K)\bigcap J(K_0^{\sep})=\Theta_d(K_0^{\sep}).$$
By our assumption, there is $\sigma \in \Gal(K)$ such that
$\sigma (\a)=N\a$ or $-N\a$ for all $\a \in J[n]$. This implies that
$\sigma (\b)=N\b$ or $-N\b$. 
It follows from Theorem \ref{ThetaD}  that  $2\b=0$ in $J(K)$. 
\end{proof}

\begin{ex}
Suppose that $K$ is the field $\C$ of complex numbers, $g=2$ and $\CC$ is the genus $2$ curve
$$y^2=x^5-x+1.$$  Let us put $N=2$. Then $d(N)=2$.
Let $n=\ell$ be an {\sl odd} prime.  Then $\Z/n\Z$ is the prime field $\F_{\ell}$. 
Results of L. Dieulefait \cite[Th. 5.8 on pp. 509--510]{Dieulefait}
 and Serre's Modularity Conjecture \cite{SerreDuke} that was proven by C. Khare and J.-P. Wintenberger \cite{Khare} imply
 that $\tilde{G}_{\ell,J,K_0}$  is ``as large as possible"; 
in particular, it contains all the homotheties $\F_{\ell}^{*}\cdot\II_{\ell}$. 
This implies that $\tilde{G}_{\ell,J,K_0}$ contains $2\cdot\II_{\ell}$, since $\ell$ is odd. 
It follows from Corollary \ref{tor} that $\Theta_1=\CC(\C)$ does {\sl not} contain points of  order $\ell$ for {\sl all odd primes} $\ell$.

Actually, using his algorithm mentioned above, B. Poonen had already checked that the only torsion points on 
this curve  are the Weierstrass points $\W_i$ (of order 2) and $\infty$ (of order 1) \cite[Sect. 14]{PoonenEXP}.

Notice that the Galois group of $x^5-x+1$ over $\Q$ is the full symmetruc group $\ST_5$. This implies that the ring of $\C$-endomorphisms of $J$ coincides with $\Z$ \cite{ZarhinMRL00}. In particular, $J$ is an absolutely simple abelian surface.
\end{ex}

\begin{thm}
\label{genericT}
 Suppose that  $g>1$,  $K_0=\Q, K=\C$ and $\alpha_1,  \dots , \alpha_{2g+1} \in \C$ are algebraically independent (transcendental) elements of $\C$
 (i.e., $$\CC:y^2=\prod_{i=1}^{2g+1}(x-\alpha_i)$$ is a {\sl generic hyperelliptic curve}). Then:
\begin{itemize}
 \item[(i)]
$\Theta_{[2g/3]}(\C)$ does {\sl not} contain nonzero points of {\sl odd} order.
\item[(ii)]
All $2$-power torsion points in 
$\Theta_{[g/2]}(\C)$ have  order 1 or 2.
\end{itemize}
\end{thm}

We will prove Theorems \ref{genericT}  in Section \ref{symplectic}.

\begin{rem}

 Let $K_0, K=\C$ and $\CC$ be as in Theorem \ref{genericT}. 
\begin{itemize}
 \item [(i)]
B. Poonen and M. Stoll \cite[Th. 7.1]{PoonenStoll} proved that the only torsion points on this generic curve are the Weierstrass points $\W_i$
 (of order 2) and $\infty$ (of order 1). 
 \item [(ii)]
Let $s_1, \dots , s_{2g+1} \in \C$ be the corresponding basic symmetric functions in $\alpha_1, \dots, \alpha_{2g+1}$ and 
let us consider the (sub)field
$$L:=\Q(s_1, \dots , s_{2g+1}\subset \Q(\alpha_1, \dots, \alpha_{2g+1})=K_0.$$
Then $f(x)$ lies in $L[x]$ and its Galois group over $L$ is the full symmetric group $\ST_{2g+1}$. 
This implies that the ring of $\C$-endomorphisms of $J$ coincides with $\Z$ \cite{ZarhinMRL00}. 
In particular, $J$ is an absolutely {\sl simple} abelian variety. (Of course, this result is well known.)
It follows from the generalized Manin-Mumford conjecture (also proven by M. Raynaud \cite{Raynaud2}) 
that the set of torsion points on $\Theta_d(\C)$ is finite for all $d<g$.

\end{itemize}
\end{rem}

\section{Abelian varieties with big $\ell$-adic Galoid images}
\label{symplectic}

We need to recall some basic facts about fields of definition of torsion points on abelian varieties.

Recall that a positive integer $n$ is {\sl not} divisible by $\fchar(K)$ and the rank $2g$ free $\Z/n\Z$-module  $J[n]$ lies in $J(K^{\sep})$. Clearly, all $n$th roots of unity of $K$ lie in $K^{\sep}$. We write $\mu_n$ for the order $n$ cyclic multiplicative group of $n$th roots of unity in $K^{\sep}$. We write $K(\mu_n)\subset K^{\sep}$ for the $n$th {\sl cyclotomic} field extension of $K$ and
$$\chi_n: \Gal(K) \to (\Z/n\Z)^{*}$$
for the $n$th {\sl cyclotomic} character that defines the Galois action on all $n$th roots of unity. The Galois group $\Gal(K(\mu_n)/K)$  of the abelian extension $K(\mu_n/K$ is canonically isomorphic to the image
$$\chi_n(\Gal(K))\subset (\Z/n\Z)^{*}=\Gal(\Q(\mu_n)/\Q);$$
the equality holds if and only if the degree $[K(\mu_n):K]$ coincides with $\phi(n)$ where $\phi$ is the {\sl Euler function}. For example, if $K$ is the field $\Q$ of rational numbers then for all $n$
$$[\Q(\zeta_n):\Q]=\phi(n), \ \chi_n(\Gal(\Q))= (\Z/n\Z)^{*}.$$

The Jacobian $J$ carries the canonical principal polarization that is defined over $K_0$ and gives rise to a nondegenerate alternating bilinear form (Weil-Riemann pairing)
$$\bar{e}_n: J[n] \times J[n] \to \Z/n\Z$$
such that for all $\sigma \in \Gal(K)$ and $\a_1, \a_2 \in J[n]$ we have
$$\bar{e}_n(\sigma\a_1,\sigma\a_2)=
\chi_n(\sigma)\cdot
 \bar{e}_n(\a_1,\a_2).$$
(Such a form is defined uniquely up to multiplication by an element of $\Z/n\Z$ and depends on a choice between of an isomorphism between $\mu_n$ and $\Z/n\Z$.)
 
Let
$$\Gp(J[n],\bar{e}_n)\subset \Aut_{\Z/n\Z}(J[n])$$ be
the group of symplectic similitudes of $\bar{e}_n$ that consists of all automorphisms
$u$ of $J[n]$ such that there exists a constant $c = c(u) \in (\Z/n\Z)^{*}$ such that
$$\bar{e_n}\bar{e}_n(u\a_1,u\a_2) = c(u)\cdot\bar{e_n}\bar{e}_n(\a_1,\a_2)
\ \forall \a_1,\a_2 \in J[n].$$
The map
$$\mathrm{mult}_n: \Gp(J[n],\bar{e}_n) \to (\Z/n\Z)^{*}, \ u\mapsto c(u)$$
is a surjective group homomorphism, whose kernel coincides with the {\sl symplectic group} 
$$\Sp(J[n],\bar{e}_n)\cong \Sp_{2g}(\F_{\ell})$$ of $\bar{e}_n$. Both $\Sp(J[n],\bar{e}_n)$ and the group of homotheties $(\Z/n\Z)\II_n$ are subgroups of $\Gp(J[n],\bar{e}_n)$. The Galois-equivariance of the Weil-Riemann pairing implies that
$$\tilde{G}_{n,J,K_0}\subset \Gp(J[n],\bar{e}_n)\subset \Aut_{\Z/n\Z}(J[n]).$$
It is also clear that for each $\sigma \in \Gal(K)$
$$\chi_n(\sigma)=c(\rho_{n,J.K_0}(\sigma))=\mathrm{mult}_n(\rho_{n,J.K_0}(\sigma)) \in (\Z/n\Z)^{*}.$$
Since $\Sp(J[n],\bar{e}_n)=\ker(\mathrm{mult}_n)$, we obtain the following useful assertion.

\begin{lem}
\label{use}

Let us assume that  $\chi_n(\Gal(K))= (\Z/n\Z)^{*}$ (E.g., $K=\Q$ or the field $\Q(t_1, \dots  t_d)$ of rational functions in $d$ independent varibles over $\mathbb{\Q}$.)

Suppose that $\tilde{G}_{n,J,K_0}$ contains $\Sp(J[n],\bar{e}_n)$. Then $\tilde{G}_{n,J,K_0}=\Gp(J[n],\bar{e}_n)$. In particular,
 $\tilde{G}_{n,J,K_0}$ contains the whole group of homotheties $(\Z/n\Z)^{*}\cdot \II_n$.
\end{lem}

\begin{ex}
\label{generic}
Let $K_0$, $K=\CC$ and $\CC$ be as in Theorem \ref{genericT},
 i.e., $\CC$ is a {\sl generic hyperelliptic curve}.
 \begin{itemize}
  \item[(i)]
B. Poonen and M. Stoll proved \cite[Proof of Th. 7.1]{PoonenStoll} that if $n$ is odd then $\tilde{G}_{n,J,K_0}$ contains $\Sp(J[n],\bar{e}_n)$.
It follows from Lemma \ref{use} that $\tilde{G}_{n,J,K_0}=\Gp(J[n],\bar{e}_n)$ for all {\sl odd} $n$. In particular, 
it contains $(\Z/n\Z)^{*}\cdot \II_n$
and therefore contains $2\cdot\II_n$.

\item[(ii)]
Let us assume that $n=2^e$ is a  a power of $2$. J. Yelton \cite{Yelton} proved that that $\tilde{G}_{n,J,K_0}$ contains 
the level 2 {\sl congruence subgroup} $\Gamma(2)$ of 
$\Sp(J[n],\bar{e}_n)$ defined by the condition
$$\Gamma(2)=\{g \in \Sp(J[n],\bar{e}_n)\mid g\equiv\II_n \bmod 2\} \lhd \Sp(J[n],\bar{e}_n).$$
Let us consider 
the level 2 {\sl congruence subgroup} $G\Gamma(2)$ of 
$\Gp(J[n],\bar{e}_n)$ defined by the condition
$$G\Gamma(2)=\{g \in \Gp(J[n],\bar{e}_n)\mid g\equiv\II_n \bmod 2\} \lhd \Gp(J[n],\bar{e}_n).$$
Clearly, $G\Gamma(2)$  contains $3\cdot\II_n$ while the intersection of  $G\Gamma(2)$ and $\Sp(J[n],\bar{e}_n)$
coincides with $\Gamma(2)$. The latter means that $\Gamma(2)$ coincides with the kernel of the restriction of
$\mathrm{mult}_n$ to $G\Gamma(2)$.
In addition,one may easily check that
$$\mathrm{mult}_n(G\Gamma(2))=(\Z/n\Z)^{*}=\mathrm{mult}_n(\Gp(J[n],\bar{e}_n),$$
since
$$(\Z/n\Z)^{*}=\{c \in \Z/n\Z)\mid c \equiv 1 \bmod 2\}.$$
This implies that 
$\tilde{G}_{n,J,K_0}$ contains $G\Gamma(2)$.
In particular, $\tilde{G}_{n,J,K_0}$ contains $3\cdot\II_n$. (See also  \cite[Proof of Th. 7.1]{PoonenStoll}.)
\end{itemize}
\end{ex}

\begin{proof}[Proof of Theorem \ref{genericT}]
 Recall that $d(2)=[2g/3]$. Combining Theorem \ref{tor} (with $N=2$ and any {\sl odd} $n$) with Example \ref{generic}(i),
 we conclude that 
$\Theta_{[2g/3]}(\C)$ does {\sl not} contain nonzero points of  odd order $n$. This proves (i).

Recall that $d(3)=[2g/4]=[g/2]$.
Combining Theorem \ref{tor} (with $N=3$ and $n=2^e$) with Example \ref{generic}(ii), we conclude that all $2$-power torsion points in 
$\Theta_{[g/2]}(\C)$ are points of order 1 or 2.
\end{proof}

The rest of this paper is devoted to the proof of the following result.

\begin{thm}
\label{notorsion}
Let $K_0$ be the field $\Q$ of rational numbers, $K=\C$ the field of complex numbers. Suppose that $g>1$.
 Let $S$ be a non-empty set of odd primes such that for all $\ell\in S$  the image $\tilde{G}_{\ell,J,K_0}=\Gp(J[\ell],\bar{e}_{\ell})$.

If $n>1$ is a positive odd integer, all whose prime divisors lie in $S$ then $\Theta_{[2g/3]}(\C)$ does not contain nonzero points of order dividing $n$.
\end{thm}

Let us start with the following elementary observation
on Galois properties of torsion points on $J$.

\begin{rem}
\label{commutator}
\begin{enumerate}
\item[(i)]
Let $\tilde{G}_n$ be the {\sl derived subgroup} $[\tilde{G}_{n,J,K_0}, \tilde{G}_{n,J,K_0}]$ of $\tilde{G}_{n,J,K_0}$.  Then $\tilde{G}_n$ is a normal subgroup of finite index in $\tilde{G}_{n,J,K_0}$. Let $K_{0,n}\subset K_0^{\sep}$ be the finite Galois extension of $K_0$ such that the absolute Galois (sub)group $\Gal(K_{0,n})\subset \Gal(K_0)$ coincides with the {\sl preimage}
 $$\rho_{n,J}^{-1}(\tilde{G}_n))\subset \rho_{n,J}^{-1}(\tilde{G}_{n,J,K_0})=\Gal(K).$$
 We have
 $$\tilde{G}_{n,J,K_{0,n}}=\rho_{n,J}(\Gal(K_{0,n}))=$$
 $$\tilde{G}_n=[\tilde{G}_{n,J,K_0}, \tilde{G}_{n,J,K_0}] \subset
 [\Gp(J[n],\bar{e}_n),\Gp(J[n],\bar{e}_n)]\subset \Sp(J[n],\bar{e}_n).$$
 This implies that
 $$\tilde{G}_{n,J,K_{0,n}}\subset \Sp(J[n],\bar{e}_n).$$
 Let $m>1$ be an integer dividing $n$. The inclusion of Galois modules $J[m]\subset J[n]$ induces the {\sl surjective} group homomorphisms
 $$\tilde{G}_{n,J,K_0}\twoheadrightarrow \tilde{G}_{m,J,K_0}, \ 
\tilde{G}_{n,J,K_{0,n}}\twoheadrightarrow \tilde{G}_{m,J,K_{0,n}}\subset \tilde{G}_{m,J,K_0};$$
the latter homomorphism coincides with the restriction of the former one to the (derived) subgroup $\tilde{G}_{n,J,K_{0,n}}\subset \tilde{G}_{n,J,K_0}$.
This implies that  $$\tilde{G}_{m,J,K_{0,n}}=[\tilde{G}_{m,J,K_0}, \tilde{G}_{m,J,K_0}]$$
is the derived subgroup of $\tilde{G}_{m,J,K_0}$. In addition,
$$\tilde{G}_{m,J,K_{0,n}}= 
[\tilde{G}_{m,J,K_0}, \tilde{G}_{m,J,K_0}]\subset
 [\Gp(J[m],\bar{e}_m),\Gp(J[m],\bar{e}_m)]\subset 
 \Sp(J[m],\bar{e}_m).
 $$

 \item[(ii)]
 Recall that $g \ge 2$.
 Now assume that $m=\ell$ is an odd prime dividing $n$. Then $\Sp(J[\ell],\bar{e}_{\ell})$ is {\sl perfect}, i.e., coincides with its own derived subgroup. Assume also that
   $\tilde{G}_{\ell,J,K_0})$ contains $\Sp(J[\ell],\bar{e}_{\ell})$. Then
$$\Sp(J[\ell],\bar{e}_{\ell})\supset \tilde{G}_{\ell,J,K_{0,n}}=$$
$$[\tilde{G}_{\ell,J,K_0}),\tilde{G}_{\ell,J,K_0})\supset 
[\Sp(J[\ell],\bar{e}_{\ell}),\Sp(J[\ell],\bar{e}_{\ell})]=\Sp(J[\ell],\bar{e}_{\ell})$$
and therefore
$$\tilde{G}_{\ell,J,K_{0,n}}=\Sp(J[\ell],\bar{e}_{\ell}).$$
 \end{enumerate}
\end{rem}

We will also need the following result  about closed subgroups of {\sl symplectic groups} 
over the ring $\Z_{\ell}$ of $\ell$-adic integers (\cite[pp. 52--53]{SerreMFV}, \cite[Th. 1.3]{Vasiu}).

\begin{lem}
\label{sur}
Let $g \ge 2$ be an integer and $\ell$ an odd prime. Let $G$ be a closed subroup of $\Sp(2g,\Z_{\ell})$ such that the corresponding reduction map $G \to \Sp(2g,\Z/\ell\Z)$ is surjective.
Then $G=\Sp(2g,\Z_{\ell})$.
\end{lem}

\begin{proof}
The result follows from  \cite[Theorem 1.3 on pp. 326--327]{Vasiu} applied to 
$$p=q=\ell, k=\F_{\ell}, W(k)=\Z_{\ell}, G=\Sp_{2g}.$$
\end{proof}

\begin{cor}
Let $g \ge 2$ be an integer and $\ell$ an odd prime. Then for each positive integer $i$ the group $\Sp_{2g}(\Z/\ell^i\Z)$ is perfect.
\end{cor}

\begin{proof}
The case $i=1$ is well known. Let $i\ge 1$ be an integer. It is also well known that the reduction modulo $\ell^i$ map
$$\mathrm{red}_i: \Sp_{2g}(\Z_{\ell})\to  \Sp_{2g}(\Z/\ell^i\Z)$$
is a {\sl surjective} group homomorphism. This implies that 
the reduction modulo $\ell$ map
$$\mathrm{\overline{red}}_{i,1}: \Sp_{2g}(\Z/\ell^i\Z)\to  \Sp_{2g}(\Z/\ell\Z)$$
is also a {\sl surjective} group homomorphism. Clearly, 
$\mathrm{red}_1$ coincides with the composition $\mathrm{\overline{red}}_{i,1}\circ \mathrm{red}_i$.

Suppose that  $\Sp_{2g}(\Z/\ell^i\Z)$ is {\sl not} perfect and let
$$H:=[\Sp_{2g}(\Z/\ell^i\Z),  \Sp_{2g}(\Z/\ell^i\Z)]$$ be the derived subgroup of $\Sp_{2g}(\Z/\ell^i\Z)$.
 Since $\Sp(2g,\Z/\ell\Z)$ is perfect, i.e., coincides with its derived subgroup, $$\mathrm{\overline{red}}_{i,1}(H) =\Sp(2g,\Z/\ell\Z).$$
Now the closed subgroup
$$G:= \mathrm{red}_i^{-1}(H)\subset \Sp_{2g}(\Z_{\ell})$$
maps surjectively on $\Sp_{2g}(\Z/\ell\Z)$ but does {\sl not} coincide with 
$\Sp_{2g}(\Z_{\ell})$, because $H$ is a {\sl proper} subgroup of $\Sp_{2g}(\Z/\ell^i\Z)$ and $\mathrm{\overline{red}}_{i,1}$ is surjective. 
This contradicts to Lemma \ref{sur}, which proves the desired perfectness.
\end{proof}

The following lemma will be proven at the end of this section.

\begin{lem}
\label{bigN}
Suppose that $g>1$.
Suppose that $n>1$ is an odd integer that is not divisible by $\fchar(K)$.  If for all primes $\ell$ dividing $n$ the image
$\tilde{G}_{\ell,J,K_0}$ contains $\Sp(J[n],\bar{e}_{\ell})$
then
$\tilde{G}_{n,J,K_0}$ contains $\Sp(J[n],\bar{e}_n)$.

In addition, if $K_0$ is the field $\Q$ of rational numbers then 
$$\tilde{G}_{n,J,K_0}= \Gp(J[n],\bar{e}_n).$$
\end{lem}

\begin{rem}
\label{useuse}
Thanks to Lemma \ref{use}, the second assertion of Lemma \ref{bigN} follows from the first one.
\end{rem}

\begin{proof}[Proof of Theorem \ref{notorsion}]
Recall that $\Gp(J[\ell],\bar{e}_{\ell})$ contains $\Sp(J[\ell],\bar{e}_{\ell})$. It follows from Lemma \ref{bigN} that 
$$\tilde{G}_{n,J,K_0}= \Gp(J[n],\bar{e}_n).$$
This implies that $\tilde{G}_{n,J,K_0}$ contains $2\cdot\II_n$, because it contains the whole $(\Z/n\Z)^{*}\II_n$. 
It follows from Corollary \ref{tor} that  $\CC(K)$ does not contain points of order $n$.
\end{proof}

\begin{proof}[Proof of Lemma \ref{bigN}]
First, let us do the case when $n$ is a power of an {\sl odd} prime $\ell$.

 Let $\ell \ne \fchar(K)$ be a prime. Let $T_{\ell}(J)$ be the $\ell$-adic Tate module of $J$ that is the projective limit of $J[\ell^i]$ where the transition maps $J[\ell^{i+1}] \to J[\ell^i]$ are multiplications by $\ell$. It is well known that $T_{\ell}(J)$ is a free $\Z_{\ell}$-module of rank $2g$, the Galois actions on $J[\ell^i]$'s are glued together to the continuous group homomorphism
$$\rho_{\ell,J,K_0}:\Gal(K)\to \Aut_{\Z_{\ell}}(T_{\ell}(J))$$
such that the canonical isomorphisms of $\Z_{\ell}$-modules
$$T_{\ell}(J)/\ell^i T_{\ell}(J)=J[\ell^i]$$
become isomorphisms of Galois modules. (Recall that $\Z/\ell^i\Z=\Z_{\ell}/\ell^i\Z_{\ell}$.) The polarization $\lambda$ gives rise to the alternating perfect/unimodular $\Z_{\ell}$-bilinear form
$$e_{\ell}: T_{\ell}(J)\times T_{\ell}(J) \to \Z_{\ell}$$
such that for each $\sigma \in \Gal(K)$
$$e_{\ell}(\rho_{\ell}(\sigma)(v_1),\rho_{\ell}(\sigma)(v_2))=\chi_{\ell}(\sigma)\cdot  e_{\ell}(v_1,v_2) \ \forall \ v_1,v_2\in T_{\ell}(J).$$
Here 
$$\chi_{\ell}:\Gal(K) \to \Z_{\ell}^{*}$$
is the (continuous) {\sl cyclotomic} character of $\Gal(K)$ characterized by the property
$$\chi_{\ell}(\sigma)\bmod \ell^i=\bar{\chi}_{\ell^i}(\sigma) \ \forall \ i.$$
This implies that
$$G_{\ell,J,K_0}=\rho_{\ell,J,K_0}\Gal(K)) \subset \Gp(T_{\ell}(J),e_{\ell}$$
where 
$$\Gp(T_{\ell}(J),e_{\ell})\subset \Aut_{\Z_{\ell}}(T_{\ell}(J))$$
is the group of symplectic similitudes of $e_{\ell}$. Clearly,
$\Gp(T_{\ell}(J),e_{\ell})$ contains the corresponding {\sl symplectic group}
$$\Sp(T_{\ell}(J),e_{\ell})\cong \Sp_{2g}(\Z_{\ell})$$ 
and the subgroup of {\sl homotheties/scalars} $\Z_{\ell}^{*}$. It is also clear that the {\sl derived subgroup}
$[\Gp(T_{\ell}(J),e_{\ell}), \Gp(T_{\ell}(J),e_{\ell})]$ lies in $\Sp(T_{\ell}(J),e_{\ell})$.

 For each $n=\ell^i$ the {\sl reduction  map} modulo $\ell^i$ sends 
$\Gp(T_{\ell}(J),e_{\ell})$ {\sl onto} $\Gp(J[\ell^i],\bar{e}_{\ell^i})$,
 $\Sp(T_{\ell}(J),e_{\ell})$ {\sl onto} $\Sp(J[\ell^i],\bar{e}_{\ell^i})$ and
 $\Z_{\ell}^{*}$ {\sl onto} $(\Z/\ell^i\Z)^{*}\cdot\II_{\ell^i}$. In particular, 
if $\ell$ is {\sl odd} then the {\sl scalar} $2\in \Z_{\ell}^{*}$ goes to 
$$2\cdot\II_{\ell^i}\in \Gp(J[\ell^i],\bar{e}_{\ell^i}).$$
As for $G_{\ell,J,K_0}$, its image under the reduction map modulo $\ell^i$ coincides with $\tilde{G}_{\ell^i,J,K_0}$.
It is known \cite{SerreAb} that $G_{\ell,J,K_0}$ is a compact $\ell$-adic Lie subgroup in $\Gp(T_{\ell}(J),e_{\ell})$ and therefore is a {\sl closed subgroup} of $\Gp(T_{\ell}(J),e_{\ell})$ with respect to $\ell$-adic topology. Clearly, the intersection 
$$G_{\ell}:=G_{\ell,J,K_0}\bigcap \Sp(T_{\ell}(J),e_{\ell})$$
is a closed subgroup of $\Sp(T_{\ell}(J),e_{\ell})$. In addition,
the {\sl derived subgroup} of $G_{\ell,J,K_0}$
$$[G_{\ell,J,K_0},G_{\ell,J,K_0}]\subset G_{\ell,J,K_0} \bigcap 
[\Gp(T_{\ell}(J),e_{\ell}), \Gp(T_{\ell}(J),e_{\ell})]\subset G_{\ell,J,K_0}\bigcap \Sp(T_{\ell}(J),e_{\ell})=G_{\ell},$$
i.e.,
$$[G_{\ell,J,K_0},G_{\ell,J,K_0}]\subset G_{\ell}.$$

Let us assume that $\ell$ is {\sl odd} and $\tilde{G}_{\ell,J,K_0}$ contains $\Sp(J[\ell],\bar{e}_{\ell})$. Then the reduction modulo $\ell$ of $[G_{\ell,J,K_0},G_{\ell,J,K_0}]$ contains the {\sl derived subgroup} $[\Sp(J[\ell],\bar{e}_{\ell}),\Sp(J[\ell],\bar{e}_{\ell})]$.
Since our assumptions on $g$ and $\ell$ imply that the group  $\Sp(J[\ell],\bar{e}_{\ell})$ is {\sl perfect}, i.e.,
$$[\Sp(J[\ell],\bar{e}_{\ell}),\Sp(J[\ell],\bar{e}_{\ell})]=\Sp(J[\ell],\bar{e}_{\ell}),$$
the reduction modulo $\ell$ of $[G_{\ell,J,K_0},G_{\ell,J,K_0}]$ contains $\Sp(J[\ell],\bar{e}_{\ell})$. This implies that the reduction modulo $\ell$ of $G_{\ell}$ also contains $\Sp(J[\ell],\bar{e}_{\ell})$. Since $G_{\ell}$ is a (closed) subgroup of $\Sp(T_{\ell}(J),e_{\ell})$, its reduction modulo $\ell$ actually {\sl coincides} with $\Sp(J[\ell],\bar{e}_{\ell})$.
It follows from Lemma \ref{sur} that 
$$G_{\ell}=\Sp(T_{\ell}(J),e_{\ell}).$$ 
In particular, the reduction of $G_{\ell}$ modulo $\ell^i$ coincides with $\Sp(J[\ell^i],\bar{e}_{\ell^i})$ for all positive integers $i$. 
Since $G_{\ell,J,K_0}$ contains $G_{\ell}$, its reduction modulo $\ell^i$ contains $\Sp(J[\ell^i],\bar{e}_{\ell^i})$. This means that
$\tilde{G}_{\ell^i,J,K_0}$ contains $\Sp(J[\ell^i],\bar{e}_{\ell^i})$ 
for all positive $i$. This proves Lemma \ref{bigN} for all $n$ that are powers of an odd prime $\ell$.

Now let us consider the general case.
 So, $n>1$ is an odd integer. 
 Let $S$ be the (finite nonempty) set of prime divisors $\ell$ of $n$ and $n=\prod_{\ell\in S}\ell^{\mathbf{d}(\ell)}$
 where all $\mathbf{d}(\ell)$ are positive integers. Using Remark \ref{commutator}, we may replace if necessary $K_0$ by $K_{0,n}$ and assume that 
 $$\tilde{G}_{\ell,J,K_0}=\Sp(J[\ell],\bar{e}_{\ell})$$
 for all $\ell \in S$. The already proven case of prime powers tells us that $$\tilde{G}_{\ell^{d(\ell)},J,K_0}=
 \Sp\left(J\left[\ell^{\mathbf{d}(\ell)}\right],\bar{e}_{\ell}\right)$$
 for all $\ell \in S$.
 On the other hand, we have
$$\Z/n\Z=\oplus_{\ell\in S}\Z/\ell^{\mathbf{d}(\ell)}\Z, \ J[n]=\oplus _{\ell\in S} J\left[\ell^{\mathbf{d}(\ell)}\right],$$
$$\Gp(J[n], \bar{e}_n)=
\prod_{\ell\in S} \Gp\left(J\left[\ell^{\mathbf{d}(\ell)}\right], \bar{e}_{\ell^{\mathbf{d}(\ell)}}\right), \ \Sp(J[n], \bar{e}_n)=
\prod_{\ell\in S}\Sp\left(J\left[\ell^{\mathbf{d}(\ell}\right], \bar{e}_{\ell^{\mathbf{d}(\ell)}}\right),$$
$$\tilde{G}_{n,J,K_0} \subset \prod_{\ell\in S}\tilde{G}_{\ell^{\mathbf{d}(\ell)},J,K_0}= 
\prod_{\ell\in S}\Sp(J[\ell^{\mathbf{d}(\ell}], \bar{e}_{\ell^{\mathbf{d}(\ell)}}).$$
Recall that the group homomorphisms 
$$\tilde{G}_{n,J,K_0} \to \tilde{G}_{\ell^{\mathbf{d}(\ell)},J,K_0}=
\Sp\left(J\left[\ell^{\mathbf{d}(\ell}\right], \bar{e}_{\ell^{\mathbf{d}(\ell)}}\right)$$
(induced by the inclusion of the Galois modules
$J\left[\ell^{\mathbf{d}(\ell)}\right]\subset J[n]$)
are {\sl surjective}. We want to use {\sl Goursat's Lemma} and {\sl Ribet's Lemma} \cite[Sect. 1.4]{SerreGroups}, in order to prove that the subgroup
$$\tilde{G}_{n,J,K_0}\subset \prod_{\ell\in S}\Sp(J[\ell^{\mathbf{d}(\ell}], \bar{e}_{\ell^{\mathbf{d}(\ell)}})$$
coincides with the whole product. In order to do that, we need to check that simple finite groups that are quotients of 
$\Sp(J[\ell^{\mathbf{d}(\ell}]$'s are {\sl mutually nonisomorphic} for {\sl different} $\ell$. Recall that
$$\Sp\left(J\left[\ell^{\mathbf{d}(\ell}\right], \bar{e}_{\ell^{\mathbf{d}(\ell)}})\right) \cong 
\Sp_{2g}\left(\Z/\ell^{\mathbf{d}(\ell)}\Z\right)$$
and therefore is {\sl perfect}.  Therefore, all its simple quotients are also perfect, i.e., are finite simple nonabelian groups.
Clearly, the only simple nonabelian quotient of $\Sp_{2g}\left(\Z/\ell^{\mathbf{d}(\ell)}\Z\right)$ is 
$$\Sigma_{\ell}:=\Sp_{2g}(\Z/\ell\Z)/\{\pm 1\}.$$
However, the groups $\Sigma_{\ell}$ are perfect and mutually nonisomorphic for distinct $\ell$ \cite{Artin1,Artin2}. This ends the proof.
\end{proof}

\begin{rem}

Remark \ref{commutator},
 Lemmas \ref{use} and \ref{bigN}, and their proofs  remain true if one replaces the jacobian $J$ by any principally polarized
 $g$-dimensional abelian variety $A$ over $K_0$ with $g \ge 2$.
\end{rem}

\end{document}